\documentclass[12pt,reqno]{amsart}

\usepackage{geometry}
\usepackage{hhline}

\usepackage{mathtools}

\usepackage{mathdots}

\usepackage{color}

\usepackage{pdfsync}

\usepackage{enumitem}

\usepackage{wasysym}

\usepackage{amssymb}

\usepackage{mathrsfs}

\usepackage{bbm}

\DeclareMathAlphabet{\mathpzc}{OT1}{pzc}{m}{it}

\usepackage[all]{xy}

\usepackage{tikz}
\usetikzlibrary{arrows,matrix,decorations.pathmorphing,decorations.pathreplacing,positioning,shapes.geometric,shapes.misc,decorations.markings,decorations.fractals,calc,patterns}

\usepackage{graphicx}

\usepackage{float}

\usepackage[bottom]{footmisc}

\usepackage{moreenum}

\usepackage{makecell}

\entrymodifiers={+!!<0pt,\fontdimen22\textfont2>}

\usepackage{scalerel,stackengine}
\stackMath
\newcommand\newcheck[1]{%
\savestack{\tmpbox}{\stretchto{%
  \scaleto{%
    \scalerel*[\widthof{\ensuremath{#1}}]{\kern-.6pt\bigwedge\kern-.6pt}%
    {\rule[-\textheight/2]{1ex}{\textheight}}
  }{\textheight}%
}{0.5ex}}%
\stackon[1pt]{#1}{\scalebox{-1}{\tmpbox}}%
}
\stackMath
\newcommand\newhat[1]{%
\savestack{\tmpbox}{\stretchto{%
  \scaleto{%
    \scalerel*[\widthof{\ensuremath{#1}}]{\kern-.6pt\bigwedge\kern-.6pt}%
    {\rule[-\textheight/2]{1ex}{\textheight}}
  }{\textheight}%
}{0.5ex}}%
\stackon[1pt]{#1}{\scalebox{1}{\tmpbox}}%
}

\def\cA{\mathscr{A}}

\def\cD{\mathscr{D}}
\def\cE{\mathscr{E}}

\def\cI{\mathscr{I}}

\def\cM{\mathscr{M}}

\def\BH{\mathbb{H}}

\def\BZ{\mathbb{Z}}

\def\Bk{\mathbbm{k}}

\def\fr{\mathfrak{r}}

\mathchardef\mhyphen="2D

\def\adots{\mathinner{\mkern1mu\raise1.0pt\vbox{\kern7.0pt\hbox{.}}\mkern2mu\raise5.0pt\hbox{.}\mkern2mu\raise9.0pt\hbox{.}\mkern1mu}}

\def\B{\operatorname{B}}

\def\Ch{\operatorname{Ch}}

\def\Coker{\operatorname{Coker}}

\def\dddots{\mathinner{\mkern1mu\raise10.0pt\vbox{\kern7.0pt\hbox{.}}\mkern2mu\raise5.3pt\hbox{.}\mkern2mu\raise1.0pt\hbox{.}\mkern1mu}}
\def\dddotssmall{\mathinner{\mkern1mu\raise7.0pt\vbox{\kern7.0pt\hbox{.}}\mkern-1mu\raise4pt\hbox{.}\mkern-1mu\raise1.0pt\hbox{.}\mkern1mu}}

\def\Diff{\operatorname{Diff}}

\def\Ext{\operatorname{Ext}}

\def\GInj{\operatorname{GInj}}

\def\H{\operatorname{H}}

\def\Hom{\operatorname{Hom}}
\def\id{\operatorname{id}}
\def\Image{\operatorname{Im}}

\def\Inj{\operatorname{Inj}}

\def\K0{\operatorname{K}_0}

\def\Ker{\operatorname{Ker}}

\def\Mod{\operatorname{Mod}}

\def\PSL2{\operatorname{PSL}_2}

\def\SL2{\operatorname{SL}_2}

\newcommand\Tensor[1]{{\underset{#1}{\otimes}}}

\def\weq{\operatorname{weq}}

\def\Z{\operatorname{Z}}

\numberwithin{equation}{section}
\renewcommand{\theequation}{\arabic{section}.\arabic{equation}}

\renewcommand{\thesubsection}{(\arabic{section}.\roman{subsection})}

\newtheorem{Lemma}{Lemma}[section]
\newtheorem{Theorem}[Lemma]{Theorem}
\newtheorem{Proposition}[Lemma]{Proposition}

\theoremstyle{definition}

\newtheorem{Remark}[Lemma]{Remark}

\theoremstyle{theorem}
\newtheorem{ThmIntro}{Theorem}

\theoremstyle{definition}
\newtheorem{DefIntro}[ThmIntro]{Definition}

\newtheorem*{bfhpg*}{}

\usepackage{amsthm}
\usepackage{thmtools}
\declaretheoremstyle[
notefont=\bfseries, notebraces={}{},
bodyfont=\normalfont,
headformat=\NUMBER~\NOTE,
headpunct={}
]{foobar}
\declaretheorem[style=foobar,numberlike=Lemma]{altbfhpg}

  {\begin{list}{}{%
    \settowidth{\labelwidth}{\textbf{#1:}}%
    \setlength{\leftmargin}{\labelwidth}\addtolength{\leftmargin}{\labelsep}}}%
  {\end{list}}

\makeatletter
\@namedef{subjclassname@2020}{%
  \textup{2020} Mathematics Subject Classification}
\makeatother

\begin{document}

\setlength{\parindent}{0pt}
\setlength{\parskip}{7pt}

\title[$Q$-shaped minimal semiinjective resolutions]{Minimal semiinjective resolutions in the $Q$-shaped derived category}

\author{Henrik Holm}

\address{Department of Mathematical Sciences, Universitetsparken 5, University of Copenhagen, 2100 Copenhagen {\O}, Denmark} 
\email{holm@math.ku.dk}

\urladdr{http://www.math.ku.dk/\~{}holm/}

\author{Peter J\o rgensen}

\address{Department of Mathematics, Aarhus University, Ny Munkegade 118, 8000 Aarhus C, Denmark}
\email{peter.jorgensen@math.au.dk}

\urladdr{https://sites.google.com/view/peterjorgensen}




\begin{abstract} 

Injective resolutions of modules are key objects of homological algebra, which are used for the computation of derived functors.  Semiinjective resolutions of chain complexes are more general objects, which are used for the computation of $\Hom$ spaces in the derived category $\cD( A )$ of a ring $A$.  Minimal semiinjective resolutions have the additional property of being unique.

\medskip
\noindent
The $Q$-shaped derived category $\cD_Q( A )$ consists of $Q$-shaped diagrams for a suitable preadditive category $Q$, and it generalises $\cD( A )$.  Some special cases of $\cD_Q( A )$ are the derived categories of differential modules, $m$-periodic chain complexes, and $N$-complexes, and there are many other possibilities. The category $\cD_Q( A )$ shares some key properties of $\cD( A )$; for instance, it is triangulated and compactly generated.

\medskip
\noindent
This paper establishes a theory of minimal semiinjective resolutions in $\cD_Q( A )$.  As a sample application, it generalises a theorem by Ringel--Zhang on differential modules.

\end{abstract}

\maketitle

\setcounter{section}{-1}
\section{Introduction}
\label{sec:introduction}

This paper generalises the theory of minimal semiinjective resolutions in $\cD( A )$, the classic derived category of a ring $A$, to $\cD_Q( A )$, the $Q$-shaped derived category.

The $Q$-shaped derived category was defined in \cite{HJ-JLMS} and \cite{HJ-TAMS}; see \cite{HJ-Abel} for a quick introduction.  The objects of $\cD_Q( A )$ are $Q$-shaped diagrams of $A$-modules where $Q$ is a suitable preadditive category.  For example, $Q$ could be given by Figure \ref{fig:linear_quiver} or Figure \ref{fig:cyclic_quiver} with the relations that $N$ consecutive arrows compose to zero for some fixed $N \geqslant 2$, and then $\cD_Q( A )$ would be the derived category of $N$-complexes or $m$-periodic $N$-complexes.  Setting $N = 2$ shows that $\cD_Q( A )$ can be specialised to $\cD( A )$, but there is a range of other choices of $Q$ enabling the construction of bespoke categories $\cD_Q( A )$, which are compactly generated triangulated categories like $\cD( A )$.

\begin{figure}
\begin{tikzpicture}[scale=1.5]
  \node at (-3,0){$\cdots$};
  \draw[->] (-2.75,0) to (-2.2,0);
  \node at (-2,0){$2$};
  \draw[->] (-1.8,0) to (-1.2,0);  
  \node at (-1,0){$1$};
  \draw[->] (-0.8,0) to (-0.2,0);    
  \node at (0,0){$0$};
  \draw[->] (0.2,0) to (0.75,0);    
  \node at (1,0){$-1$};
  \draw[->] (1.25,0) to (1.75,0);    
  \node at (2,0){$-2$};    
  \draw[->] (2.25,0) to (2.75,0);    
  \node at (3,0){$\cdots$};    
\end{tikzpicture}
\caption{A chain complex is a diagram of this form.}
\label{fig:linear_quiver}
\end{figure}

The theory of minimal semiinjective resolutions in $\cD( A )$ was developed by Avramov--Foxby--Halperin \cite{AFH}, Christensen--Foxby--Holm \cite[app.\ B]{CFH-book}, Foxby \cite[sec.\ 10]{HBF-preprints-19}, Garc\'{\i}a Rozas \cite[sec.\ 2.3 and 2.4]{Garcia-Rozas-book}, Krause \cite[app.\ B]{Krause}, and, in an abstract version, Roig \cite{Roig}.  Minimal injective resolutions of modules are a special case, and minimal semiinjective resolutions in $\cD( A )$ have a range of applications.  Chen--Iyengar and Foxby used them to investigate the small support \cite[prop.\ 2.1]{Chen-Iyengar}, \cite[rmk.\ 2.9]{HBF-Bounded-complexes}, Christensen--Iyengar--Marley used them to prove results on $\Ext$ rigidity \cite[prop.\ 3.2, prop.\ 3.4, thm.\ 5.1]{Christensen-Iyengar-Marley}, Enochs--Jenda--Xu linked them to relative homological algebra \cite[thm.\ 3.18]{Enochs-Jenda-Xu}, and Iacob--Iyengar used them to characterise regular rings \cite[prop.\ 2.10]{Iacob-Iyengar}.

Motivated by this we will develop a theory of minimal semiinjective resolutions in $\cD_Q( A )$.  We will provide different characterisations of minimal semiinjective objects in Theorem \ref{thm:15_23_25} and use them to establish the existence and uniqueness of minimal semiinjective resolutions in Theorems \ref{thm:22}, \ref{thm:27}, \ref{thm:36}.  As a sample application, we will generalise Ringel--Zhang's result \cite[thm.\ 2]{Ringel-Zhang} on differential modules; see Theorem \ref{thm:RZ}.

Before stating the main definitions and results, we fix the following for the rest of the paper.
\begin{itemize}
\setlength\itemsep{4pt}

  \item  $\Bk$ is a hereditary noetherian commutative ring.
  
  \item  $A$ is a $\Bk$-algebra.
  
  \item  $\Mod( A )$ is the category of $A$-left modules.  The full subcategory of injective modules is $\Inj( A )$.

  \item  $Q$ is a small $\Bk$-preadditive category where $Q_0$ denotes the class of objects, $Q( -,- )$ the $\Hom$ functor.  Using the terminology of \cite[Setup 1.1]{HJ-Abel}, we assume that $Q$ satisfies the conditions {\em Hom finiteness}, {\em Local boundedness}, {\em Serre functor}, {\em Strong retraction}, and {\em Nilpotence}, and we denote the {\em pseudoradical} by $\fr$.
  
  \item  $S$ denotes the Serre functor of $Q$, characterised by the existence of isomorphisms $Q( p,q ) \cong \Hom_{ \Bk }\big( Q( q,Sp ),\Bk \big)$, natural in $p$ and $q$.

\end{itemize}
The conditions on $Q$ are satisfied in the examples above where $Q$ is given by Figure \ref{fig:linear_quiver} or Figure \ref{fig:cyclic_quiver} with the relations that $N$ consecutive arrows compose to zero for some fixed $N \geqslant 2$; see \cite[1.5]{HJ-Abel}.

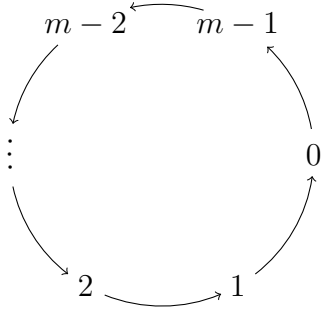
\begin{figure}
\begin{tikzpicture}[scale=2]
  \node at (0:1.0){$0$};
  \draw[->] (10:1.0) arc (10:46:1.0);
  \node at (60:1.0){$m-1$};
  \draw[->] (74:1.0) arc (74:102:1.0);
  \node at (120:1.0){$m-2$};
  \draw[->] (133:1.0) arc (133:168:1.0);
  \node at (175:1.0){$\cdot$};
  \node at (180:1.0){$\cdot$};
  \node at (185:1.0){$\cdot$};  
  \draw[->] (192:1.0) arc (192:232:1.0);
  \node at (240:1.0){$2$};
  \draw[->] (248:1.0) arc (248:293:1.0);
  \node at (300:1.0){$1$};
  \draw[->] (308:1.0) arc (308:352:1.0);
\end{tikzpicture}
\caption{An $m$-periodic chain complex is a diagram of this form.}
\label{fig:cyclic_quiver}
\end{figure}

We also need the following notation from \cite{HJ-JLMS} and \cite{HJ-TAMS}.
\begin{itemize}
\setlength\itemsep{4pt}

  \item  The category of $Q$-shaped diagrams with values in $\Mod( A )$ is
\[
  {}_{ Q,A }\!\Mod
  = \{\, \mbox{$\Bk$-linear functors $Q \xrightarrow{} \Mod( A )$} \,\}.
\]
It is a Grothendieck abelian category which generalises the abelian category of chain complexes of $A$-modules; see \cite[1.5]{HJ-Abel} and \cite[prop.\ 3.12]{HJ-JLMS}.

  \item  In ${}_{ Q,A }\!\Mod$, the $\Hom$ functor is $\Hom_{ Q,A }$, the $i$'th $\Ext$ functor is $\Ext^i_{ Q,A }$, and the full subcategory of injective objects is ${}_{ Q,A }\!\Inj$.  For $f,g$ in $\Hom_{ Q,A }( X,Y )$, we write $f \sim g$ if $f-g$ factors through an object of ${}_{ Q,A }\!\Inj$.

  \item  The class of exact objects in ${}_{ Q,A }\!\Mod$, defined in \cite[def.\ 4.1]{HJ-JLMS}, is $\cE$.  It generalises the class of exact chain complexes; see \cite[2.2 and 2.4]{HJ-Abel}.

  \item  The class of weak equivalences in ${}_{ Q,A }\Mod$, defined in \cite[prop.\ 6.3]{HJ-JLMS}, is $\weq$.  It generalises the class of quasi-isomorphisms of chain complexes; see \cite[2.3 and 2.4]{HJ-Abel}.
  
  \item  The $Q$-shaped derived category, obtained by inverting the morphisms in $\weq$, is
\[
  \cD_Q( A )
  = \weq^{ -1 }\!{}_{ Q,A }\!\Mod\!.
\]
It generalises the classic derived category of $A$.

\end{itemize}

The conditions on $Q$ are mainly due to Dell'Ambrogio--Stevenson--\v{S}\v{t}ov\'{\i}\v{c}ek \cite[thm.\ 1.6]{DSS}, while the definition of $\cD_Q( A )$ is based on the insight of Iyama--Minamoto that the key property of $Q$ which makes $\cD_Q( A )$ well behaved is the existence of a Serre functor on $Q$; see \cite{Iyama-Minamoto-1}, \cite[sec.\ 2]{Iyama-Minamoto-2}.

\begin{center}
{\bf Semiinjective objects}
\end{center}
The following is the key definition of this paper.  Note that part (i) appeared in \cite[3.1]{HJ-Abel}, and that the class $\cE^{ \perp }$ is used intensively in \cite{HJ-JLMS}, \cite{HJ-TAMS}.

\begin{DefIntro}
\label{def:A}
\begin{enumerate}
\setlength\itemsep{4pt}

  \item  A {\em semiinjective object} in ${}_{ Q,A }\!\Mod$ is an object in the class
\[
  \cE^{ \perp } =
  \{\, I \in {}_{ Q,A }\!\Mod \,|\, \Ext_{ Q,A }^1( \cE,I ) = 0 \,\}.
\]

  \item  A {\em minimal semiinjective object} in ${}_{ Q,A }\!\Mod$ is a semiinjective object whose only subobject in ${}_{ Q,A }\!\Inj$ is $0$.

  \item  A {\em semiinjective resolution} of $X$ in ${}_{ Q,A }\!\Mod$ is a weak equivalence $X \xrightarrow{} I$ with $I$ semiinjective.

  \item  A {\em minimal semiinjective resolution} of $X$ in ${}_{ Q,A }\!\Mod$ is a weak equivalence $X \xrightarrow{} I$ with $I$ minimal semiinjective.

\end{enumerate}
\end{DefIntro}

These concepts generalise (minimal) semiinjective chain complexes and (minimal) semiinjective resolutions of chain complexes; see \cite[prop.\ 2.3.14 and sec.\ 2.4]{Garcia-Rozas-book}, \cite[3.2]{HJ-Abel}.  Semiinjective chain complexes are due to B\"{o}kstedt--Neeman \cite[sec.\ 2]{Boekstedt-Neeman} (who used the term ``special complexes of injectives'') and Garc\'{\i}a-Rozas \cite[prop.\ 2.3.4]{Garcia-Rozas-book} (who used the term ``DG-injective complexes'').

One reason for the interest in semiinjective objects is that they can be used to compute $\Hom$ spaces in $\cD_Q( A )$, which are otherwise hard to access.  Each object $Y$ in ${}_{ Q,A }\!\Mod$ has a semiinjective resolution $Y \xrightarrow{} I$ by Theorem \ref{thm:22}(i).  If $X$ is also an object in ${}_{ Q,A }\!\Mod$, then
\begin{equation}
\label{equ:Hom_by_semiinjective_resolution}
  \Hom_{ \cD_Q( A ) }( X,Y )
  \cong
  \Hom_{ Q,A }( X,I )/\sim
\end{equation}
by Proposition \ref{pro:weq_property}(ii).  Equation \eqref{equ:Hom_by_semiinjective_resolution} generalises the computation of $\Hom$ spaces in $\cD( A )$ using semiinjective resolutions of chain complexes; see \cite[cor.\ 7.3.22]{CFH-book}.

\begin{center}
{\bf Minimal semiinjective objects and resolutions}
\end{center}

Our main results provide different characterisations of minimal semiinjective objects in Theorem \ref{thm:15_23_25} and establish the existence and uniqueness of minimal semiinjective resolutions in Theorems \ref{thm:22}, \ref{thm:27}, \ref{thm:36}.  Appendix \ref{app:complexes} explains how these results can be specialised to the theory of minimal semiinjective resolutions in $\cD( A )$.

\begin{ThmIntro}
\label{thm:22}
\begin{enumerate}
\setlength\itemsep{4pt}

  \item  Each $X$ in ${}_{ Q,A }\!\Mod$ has a minimal semiinjective resolution.

  \item  Each semiinjective object $I$ in ${}_{ Q,A }\!\Mod$ has the form $I = I' \oplus J'$ in ${}_{ Q,A }\!\Mod$ with $I'$ a minimal semiinjective object and $J'$ in ${}_{ Q,A }\!\Inj$.  

\end{enumerate}
\end{ThmIntro}

\begin{ThmIntro}
\label{thm:27}
If $I \xrightarrow{ i } I'$ in ${}_{ Q,A }\!\Mod$ is a weak equivalence between minimal semiinjective objects, then $i$ is an isomorphism in ${}_{ Q,A }\!\Mod$.
\end{ThmIntro}

\begin{ThmIntro}
\label{thm:36}
If $X \xrightarrow{ x } I$ and $X \xrightarrow{ x' } I'$ are minimal semiinjective resolutions in ${}_{ Q,A }\!\Mod$ then:
\begin{itemize}
\setlength\itemsep{4pt}

  \item  The diagram
\[
\vcenter{
  \xymatrix @+0.5pc {
    X \ar^{ x }[r] \ar_{ x' }[d] & I \ar@{.>}^{ i }[dl] \\
    I' \\
                    }
        }
\]
can be completed with a morphism $i$ such that $ix \sim x'$ in ${}_{ Q,A }\!\Mod$.

  \item  The morphism $i$ is unique up to equivalence under ``$\sim$''.

  \item  Each completing morphism $i$ is an isomorphism in ${}_{ Q,A }\!\Mod$.

\end{itemize}
\end{ThmIntro}

Note that the first bullet in Theorem \ref{thm:36} cannot be improved to say $ix = x'$ instead of $ix \sim x'$.

The proof of Theorem \ref{thm:22} uses the full force of the results on the $Q$-shaped derived category established in \cite{HJ-JLMS} and \cite{HJ-TAMS} as well as most of the machinery developed in this paper.  The easier Theorems \ref{thm:27} and \ref{thm:36} could have been obtained as consequences of \cite[cors.\ 1 and 2]{Roig} because Theorem \ref{thm:15_23_25}(iii) implies that our notion of minimal semiinjective resolutions is an instance of the right minimal models of \cite[sec.\ 1]{Roig}.  However, we provide short, self contained proofs for the benefit of the reader.

\begin{center}
{\bf Differential modules}
\end{center}

As a sample application of our theory, we will generalise Ringel--Zhang's result \cite[thm.\ 2]{Ringel-Zhang} on differential modules.  

A differential module over the ring $A$ is a pair $( M,\partial )$ with $M$ in $\Mod( A )$ and $M \xrightarrow{ \partial } M$ an endomorphism with $\partial^2 = 0$.  This notion was defined by Cartan--Eilenberg under the name ``modules with differentiation'', see \cite[p.\ 53]{Cartan-Eilenberg-Book}.  There is a Grothendieck abelian category $\Diff( A )$ of differential modules over $A$ in which the notions of injective and Gorenstein injective objects make sense, see \cite[sec.\ 7]{Krause}.  The homology functor $\Diff( A ) \xrightarrow{ \H } \Mod( A )$ is defined on objects by $\H( M,\partial ) = \Z( M,\partial )/\B( M,\partial )$ where $\Z( M,\partial ) = \Ker \partial$ is the cycles, $\B( M,\partial ) = \Image \partial$ the boundaries.  
\label{page:Diff}

A notable result on differential modules was proved by Ringel--Zhang in \cite[thm.\ 2]{Ringel-Zhang}.  They worked with finite dimensional differential modules over the path algebra of a finite, acyclic quiver.  We provide the following generalisation to arbitrary differential modules over a hereditary ring.

\begin{ThmIntro}
\label{thm:RZ}
Assume that $A$ is a left hereditary ring.  Then the homology functor $\Diff( A ) \xrightarrow{ \H } \Mod( A )$ induces a bijection 
\begin{equation}
\label{equ:RZ:a}
  \left\{\!
    \begin{array}{l}
      \mbox{Isomorphism classes of Gorenstein injective objects} \\[1mm]
      \mbox{without non-zero injective summands in $\Diff( A )$}
    \end{array}
  \!\right\}
  \xrightarrow{\H}
  \left\{\!
    \begin{array}{l}
      \mbox{Isomorphism classes} \\[1mm]
      \mbox{in $\Mod( A )$}      
    \end{array}
  \!\right\}.
\end{equation}
\end{ThmIntro}

To place this in a wider context, recall from \cite[sec.\ 7]{Krause} that if $\cA$ is a Grothendieck abelian category, then $\GInj \cA$, the full subcategory of Gorenstein injective objects of $\cA$, is a Frobenius category with projective-injective objects given by $\Inj \cA$, the injective objects of $\cA$.  The naive quotient category $\GInj \cA / \Inj \cA$ is triangulated by \cite[thm.\ I.2.6]{Happel-book} and is important in the context of Gorenstein approximations and Tate cohomology.
Understanding the objects of $\GInj \cA / \Inj \cA$ amounts to understanding the objects of $\GInj \cA$ up to injective summands, see \cite[thm.\ 13.7]{Hilton-book}, and this is accomplished by Theorem \ref{thm:RZ} for $\cA = \Diff( A )$.

Theorem \ref{thm:RZ} will be proved by translating the left hand set of Equation \eqref{equ:RZ:a} to the set of isomorphism classes of minimal semiinjective differential modules.  The inverse bijection to \eqref{equ:RZ:a} is induced by sending $M$ to a minimal semiinjective resolution of $( M,0 )$.  Theorem \ref{thm:RZ} is an injective analogue of \cite[cor.\ 1.4]{Wei}.

\begin{center}
{\bf Structure of the paper}
\end{center}
Section \ref{sec:lemmas} proves some preliminary results.  Section \ref{sec:main} provides different characterisations of minimal semiinjective objects in Theorem \ref{thm:15_23_25} and uses them to prove Theorems \ref{thm:22}, \ref{thm:27}, \ref{thm:36}.  Section \ref{sec:differential} proves Theorem \ref{thm:RZ}.  Appendix \ref{app:complexes} shows how our theory specialises to the theory of minimal semiinjective resolutions in $\cD( A )$.

\section{Preliminary results}
\label{sec:lemmas}

This section proves some preliminary results required to establish the theorems stated in the introduction.

\begin{Proposition}
\label{pro:weq_property}
${\;}$
\begin{enumerate}
\setlength\itemsep{4pt}

  \item  There are isomorphisms
\[
\vcenter{
  \xymatrix @+0.5pc {
    \Hom_{ Q,A }( X,I )/\sim \ar[r] & \Hom_{ \cD_Q( A ) }( X,I ), \\
                    }
        }
\]
natural with respect to $X$ in ${}_{ Q,A }\!\Mod$ and $I$ in $\cE^{ \perp }$.

  \item  If $X$ and $Y$ are in ${}_{ Q,A }\!\Mod$ and $Y \xrightarrow{} I$ is a semiinjective resolution, then there is an isomorphism
\[
\vcenter{
  \xymatrix @+0.5pc {
    \Hom_{ Q,A }( X,I )/\sim \ar[r] & \Hom_{ \cD_Q( A ) }( X,Y ). \\
                    }
        }
\]

  \item  If $X \xrightarrow{} Y$ is a weak equivalence in ${}_{ Q,A }\!\Mod$ and $I$ is in $\cE^{ \perp }$, then the induced map
\[
\vcenter{
  \xymatrix @+0.5pc {
    \Hom_{ Q,A }( Y,I )/\sim \ar[r] & \Hom_{ Q,A }( X,I )/\sim \\
                    }
        }
\]
is an isomorphism.

\end{enumerate}
\end{Proposition}

\begin{proof}
(i) Use \cite[thm.\ 6.1(b) and its proof]{HJ-JLMS} to get a ``Hovey triple'' $( {}_{ Q,A }\!\Mod,\cE,\cE^{ \perp } )$.  Then apply \cite[thm.\ 2.6]{Gillespie-BLMS}, where the Hovey triple is called ``abelian model category'', noting that the ``core'' ${}_{ Q,A }\!\Mod \cap \cE \cap \cE^{ \perp }$ is ${}_{ Q,A }\!\Inj$ by \cite[thm.\ 4.4(b)]{HJ-JLMS}.  

(ii) Compose the isomorphism from part (i) with the inverse of the isomorphism
\[
\vcenter{
  \xymatrix @+0.5pc {
    \Hom_{ \cD_Q( A ) }( X,Y ) \ar[r] & \Hom_{ \cD_Q( A ) }( X,I ) \\
                    }
        }
\]
which results from the weak equivalence $Y \xrightarrow{} I$ inducing an isomorphism in $\cD_Q( A )$.

(iii) By part (i) the morphism $X \xrightarrow{} Y$ induces a commutative square
\[
\vcenter{
  \xymatrix @+0.5pc {
    \Hom_{ Q,A }( Y,I )/\sim \ar[r] \ar[d] & \Hom_{ \cD_Q( A ) }( Y,I ) \ar[d] \\
    \Hom_{ Q,A }( X,I )/\sim \ar[r] & \Hom_{ \cD_Q( A ) }( X,I ) \\
                    }
        }
\]
where the horizontal maps are isomorphisms.  The right hand vertical map is an isomorphism because the weak equivalence $X \xrightarrow{} Y$ induces an isomorphism in $\cD_Q( A )$, so the left hand vertical map is also an isomorphism.
\end{proof}

\begin{Proposition}
\label{pro:10}
Let $0 \xrightarrow{} X' \xrightarrow{ \xi' } X \xrightarrow{ \xi } X'' \xrightarrow{} 0$ be a short exact sequence in ${}_{ Q,A }\!\Mod$.
\begin{enumerate}
\setlength\itemsep{4pt}

  \item  $\xi'$ is a weak equivalence $\Leftrightarrow$ $X''$ is in $\cE$.

  \item  $\xi$ is a weak equivalence $\Leftrightarrow$ $X'$ is in $\cE$.

\end{enumerate}
\end{Proposition}

\begin{proof}
For each $q$ in $Q_0$ the homology and cohomology functors of \cite[def.\ 7.11]{HJ-JLMS} give long exact sequences
\[
  \cdots
  \xrightarrow{}
  \BH^{ [q] }_{ i+1 }( X'' )
  \xrightarrow{}
  \BH^{ [q] }_i( X' )
  \xrightarrow{ \xi'_* }  
  \BH^{ [q] }_i( X )
  \xrightarrow{ \xi_* }  
  \BH^{ [q] }_i( X'' )
  \xrightarrow{}  
  \BH^{ [q] }_{ i-1 }( X' )
  \xrightarrow{}
  \cdots
\]
and
\[
  \cdots
  \xrightarrow{}
  \BH_{ [q] }^{ i-1 }( X'' )
  \xrightarrow{}
  \BH_{ [q] }^i( X' )
  \xrightarrow{ \xi'_* }  
  \BH_{ [q] }^i( X )
  \xrightarrow{ \xi_* }  
  \BH_{ [q] }^i( X'' )
  \xrightarrow{}  
  \BH_{ [q] }^{ i+1 }( X' )
  \xrightarrow{}
  \cdots.
\]
Combining these with \cite[thms.\ 7.1 and 7.2]{HJ-JLMS} proves the lemma.
\end{proof}

The following lemma and later parts of the paper use the notions of (special) preenvelopes, left minimal morphisms, and envelopes, see \cite[defs.\ 2.1.1 and 2.1.12]{Goebel-Trlifaj-book}.

\begin{Lemma}
\label{lem:9}
Let $E$ be in $\cE$.
\begin{enumerate}
\setlength\itemsep{4pt}

  \item  Each ${}_{ Q,A }\!\Inj$-preenvelope $E \xrightarrow{ e } J$ is a special $\cE^{ \perp }$-preenvelope.

  \item  Each ${}_{ Q,A }\!\Inj$-envelope $E \xrightarrow{ e } J$ is an $\cE^{ \perp }$-envelope.

\end{enumerate}
\end{Lemma}

\begin{proof}
(i)  Since $e$ is a ${}_{ Q,A }\!\Inj$-preenvelope in the abelian Grothendieck category ${}_{ Q,A }\!\Mod$, it is a monomorphism; see \cite[cor.\ X.4.3]{Stenstroem}.  It defines a short exact sequence $0 \xrightarrow{} E \xrightarrow{ e } J \xrightarrow{} E' \xrightarrow{} 0$, which induces an exact sequence
\[
  \Hom_{ Q,A }( J,I )
  \xrightarrow{ e^* }
  \Hom_{ Q,A }( E,I )
  \xrightarrow{}
  \Ext^1_{ Q,A }( E',I )
\]
for each $I$.  Since $J$ is in $\cE$ by \cite[thm.\ 4.4(b)]{HJ-JLMS} we have $E'$ in $\cE$ by the last part of \cite[thm.\ 4.4]{HJ-JLMS}.  If $I$ is in $\cE^{ \perp }$ then $\Ext^1_{ Q,A }( E',I ) = 0$ whence $e^*$ is an epimorphism.  We also know that $J$ is in $\cE^{ \perp }$ by \cite[thm.\ 4.4(b)]{HJ-JLMS}, so $e$ is an $\cE^{ \perp }$-preenvelope.  It is special because $E'$ is in $\cE$ which is equal to ${}^{ \perp }( \cE^{ \perp } )$ by \cite[thm.\ 4.4(b)]{HJ-JLMS}.

(ii)  Since $e$ is a ${}_{ Q,A }\!\Inj$-envelope, it is a (special) $\cE^{ \perp }$-preenvelope by part (i).  Since it is an envelope, it is a left minimal morphism.  Hence it is an $\cE^{ \perp }$-envelope.
\end{proof}

\begin{Remark}
\label{rmk:EFG}
We recall the adjoint pairs of functors
\begingroup
\[
\xymatrix
{
  {}_{ Q,A }\!\Mod
    \ar[rr]^{ E_q } &&
  \Mod( A )
    \ar@/_2.0pc/[ll]_{ F_q }
    \ar@/^2.0pc/[ll]^{ G_q }
}
\;\;\;\;\mbox{given by}\;\;\;\;
\arraycolsep=1.4pt\def\arraystretch{1.5}
\begin{array}{rcl}
  F_q( M ) & = & Q( q,- ) \Tensor{\Bk} M, \\
  E_q( X ) & = & X( q ), \\
  G_q( M ) & = & \Hom_{ \Bk }\!\big( Q( -,q ),M \big), \phantom{\rule[-12pt]{1pt}{1pt}} \\
\end{array}
\]
\endgroup
which exist for each $q$ in $Q_0$, see \cite[cor.\ 3.9]{HJ-JLMS}.

The functor $E_q$ generalises the functor sending a chain complex to its $q$'th component.  The functors $F_q$ and $G_q$ generalise the indecomposable projective and injective representations of $Q$ at $q$ known from the context of quiver representations; see \cite[def.\ 5.3]{Schiffler}.
\end{Remark}

\begin{Lemma}
\label{lem:18}
Let $\{ F_qM_q \xrightarrow{ \varphi_q } X \}_{ q \in Q_0 }$ be a family of monomorphisms in ${}_{ Q,A }\!\Mod$.  Then the induced morphism $\coprod_{ p \in Q_0 } F_pM_p \xrightarrow{ \varphi } X$ is a monomorphism.
\end{Lemma}

\begin{proof}
By definition, $\varphi$ is the unique morphism such that the following diagram is commutative for each $q$ in $Q_0$,
\[
\vcenter{
  \xymatrix @+0.5pc {
    F_qM_q \ar_{ \iota_q }[d] \ar^{ \varphi_q }[dr] \\
    \coprod_{ p \in Q_0 } F_pM_p \ar_-{ \varphi }[r] & X \lefteqn{,} \\
                    }
        }
\]
where $\iota_q$ denotes the coproduct inclusion.  The diagram can be extended as follows.
\[
\vcenter{
  \xymatrix @+0.5pc {
    & & F_qM_q \ar@<-1ex>_{ \iota_q }[d] \ar^{ \varphi_q }[dr] \\
    0 \ar[r] & \Ker \varphi \ar[r] & \coprod_{ p \in Q_0 } F_pM_p \ar_-{ \varphi }[r] \ar@<-1ex>[u] \ar@<1ex>[d] & X \\
    & & \coprod_{ p' \in Q_0 \setminus q } F_{ p' }M_{ p' } \ar@<1ex>[u]
                    }
        }
\]
Here the column is a biproduct diagram, $\varphi_q$ is a monomorphism by assumption, and the row is left exact.  Recalling that $S$ denotes the Serre functor of $Q$, we have the object $Sq$ in $Q_0$.  Applying the functor $K_{ Sq }$ of \cite[prop.\ 7.15]{HJ-JLMS} to the diagram gives the following.
\[
\vcenter{
  \xymatrix @+0.5pc {
    & & K_{ Sq }( F_qM_q ) \ar@<-1ex>_{ K_{ Sq }( \iota_q ) }[d] \ar^{ K_{ Sq }( \varphi_q ) }[dr] \\
    0 \ar[r] & K_{ Sq }( \Ker \varphi ) \ar[r] & K_{ Sq } \big( \coprod_{ p \in Q_0 } F_pM_p \big) \ar_-{ K_{ Sq }( \varphi ) }[r] \ar@<-1ex>[u] \ar@<1ex>[d] & K_{ Sq }( X ) \\
    & & K_{ Sq } \big( \coprod_{ p' \in Q_0 \setminus q } F_{ p' }M_{ p' } \big)\ar@<1ex>[u]
                    }
        }
\]
The functor $K_{ Sq }$ is additive, so the column is a biproduct diagram.  The functor $K_{ Sq }$ is a right adjoint by \cite[prop.\ 7.15]{HJ-JLMS}, hence left exact, so $K_{ Sq }( \varphi_q )$ is a monomorphism and the row is left exact.  

In the last diagram, the bottom object is
\[
  K_{ Sq } \big( \coprod_{ p' \in Q_0 \setminus q } F_{ p' }M_{ p' } \big)
  \cong K_{ Sq } \big( \prod_{ p' \in Q_0 \setminus q } F_{ p' }M_{ p' } \big)
  \cong \prod_{ p' \in Q_0 \setminus q } K_{ Sq }F_{ p' }M_{ p' } \\
  \cong \prod_{ p' \in Q_0 \setminus q } K_{ Sq }G_{ Sp' }M_{ p' } \\
  \cong 0,
\]
where the first isomorphism holds by \cite[prop.\ 3.7]{HJ-TAMS}, the second isomorphism holds since $K_{ Sq }$ is a right adjoint functor by \cite[prop.\ 7.15]{HJ-JLMS}, and the third and fourth isomorphisms hold by \cite[lem.\ 3.4]{HJ-TAMS} and \cite[lem.\ 7.28(b)]{HJ-JLMS}.  Hence $K_{ Sq }( \iota_q )$ is an isomorphism since the column is a biproduct diagram.  But $K_{ Sq }( \varphi_q )$ is a monomorphism so the commutative triangle in the diagram implies that $K_{ Sq }( \varphi )$ is a monomorphism, whence left exactness of the row implies $K_{ Sq }( \Ker \varphi ) = 0$.  Since this holds for each $q \in Q_0$, we get $\Ker \varphi = 0$ by \cite[prop.\ 7.19]{HJ-JLMS} so $\varphi$ is a monomorphism as desired.
\end{proof}

\begin{Lemma}
\label{lem:19}
Let $M \xrightarrow{ m } N$ be an essential extension in $\Mod( A )$ (that is, a monomorphism with essential image).  Then $F_qM \xrightarrow{ F_qm } F_qN$ is an essential extension in ${}_{ Q,A }\!\Mod$ for each $q$ in $Q_0$.
\end{Lemma}

\begin{proof}
The functor $F_q$ is exact by \cite[cor.\ 3.9]{HJ-JLMS} so $F_qM \xrightarrow{ F_qm } F_qN$ is a monomorphism.  Up to isomorphism, it can be written $G_pM \xrightarrow{ G_pm } G_pN$ by \cite[lem.\ 3.4]{HJ-TAMS} where $p = Sq$.  We must prove that if $X \subseteq G_pN$ has zero intersection with the image of $G_pm$, then $X$ is zero.  So let $X \xrightarrow{ \xi } G_pN$ denote the inclusion and assume that
\[
\vcenter{
  \xymatrix @+0.5pc {
    0 \ar[r] \ar[d] & X \ar^{ \xi }[d] \\
    G_pM \ar_{ G_pm }[r] & G_pN \\
                    }
        }
\]
is a pullback diagram; we must prove that $X$ is zero.  

For $r$ in $Q_0$, the functor $K_r$ of \cite[prop.\ 7.15]{HJ-JLMS} is a right adjoint, hence left exact.  It follows that there is a pullback diagram
\[
\vcenter{
  \xymatrix @+1.5pc {
    0 \ar[r] \ar[d] & K_r( X ) \ar^{ K_r( \xi ) }[d] \\
    K_r( G_pM ) \ar_{ K_r( G_pm ) }[r] & K_r( G_pN ) \lefteqn{.} \\
                    }
        }
\]
If $r \neq p$ then $K_rG_p = 0$ by \cite[lem.\ 7.28(b)]{HJ-JLMS}; in particular, $K_r( G_pN ) = 0$ whence the diagram implies $K_r( X ) = 0$.  If $r = p$ then $K_rG_p \cong \id$ by \cite[lem.\ 7.28(b)]{HJ-JLMS}; in particular, the diagram is isomorphic to a pullback diagram
\[
\vcenter{
  \xymatrix @+0.5pc {
    0 \ar[r] \ar[d] & K_r( X ) \ar[d] \\
    M \ar_{ m }[r] & N \lefteqn{.} \\
                    }
        }
\]
Since $M \xrightarrow{ m } N$ is an essential extension, this implies $K_r( X ) = 0$. 

Hence $K_r( X ) = 0$ for each $r$ in $Q_0$, so $X = 0$ as desired by \cite[prop.\ 7.19]{HJ-JLMS}.
\end{proof}

\begin{Remark}
\label{rmk:adjunction}
We recall two properties of adjoint functors.
\begin{enumerate}
\setlength\itemsep{4pt}

  \item  The adjunction isomorphism $\Hom_A( M,E_qX ) \xrightarrow{} \Hom_{ Q,A }( F_qM,X )$ maps a morphism 
\[
M \xrightarrow{ \mu } E_qX
\]
to the {\em adjoint morphism}
\[
  F_qM \xrightarrow{ \varphi } X
\]  
defined as the composition of the morphisms
\begin{equation}
\label{equ:rmk:adjunction:a}
  F_qM \xrightarrow{ F_q\mu } F_qE_qX \xrightarrow{ \varepsilon_X } X,
\end{equation}
where $\varepsilon$ is the counit of the adjoint pair $( F_q,E_q )$.

  \item  If the diagram
\begin{equation}
\label{equ:lem:pre20_21:a:a}
\vcenter{
  \xymatrix @+0.5pc {
    M \ar^{ m }[r] \ar_{ \mu }[d] & N \ar^-{ \nu }[dl] \\
    E_qX \\
                    }
        }
\end{equation}
is commutative, then so is the diagram
\begin{equation}
\label{equ:lem:pre20_21:a:b}
\vcenter{
  \xymatrix @+0.5pc {
    F_qM \ar^{ F_qm }[r] \ar_{ \varphi }[d] & F_qN \ar^-{ \psi }[dl] \\
    X \\
                    }
        }
\end{equation}
where the adjoint morphisms of $\mu$ and $\nu$ are $\varphi$ and $\psi$.

\end{enumerate}
\end{Remark}

\begin{Lemma}
\label{lem:14}
Let $q$ in $Q_0$ and $X$ in ${}_{ Q,A }\!\Mod$ be given.  Consider a morphism $M \xrightarrow{ \mu } E_qX$ with adjoint morphism $F_qM \xrightarrow{ \varphi } X$.  If $\varphi$ is a monomorphism, then $\mu$ is a monomorphism.
\end{Lemma}

\begin{proof}
Assume that $\varphi$ is a monomorphism.  Then $F_q\mu$ is a monomorphism since $\varphi$ is the composition of the morphisms in Equation \eqref{equ:rmk:adjunction:a}.  There is an exact sequence $0 \xrightarrow{} \Ker \mu \xrightarrow{} M \xrightarrow{ \mu } E_qX$, hence an exact sequence $0 \xrightarrow{} F_q\Ker \mu \xrightarrow{} F_qM \xrightarrow{ F_q\mu } F_qE_qX$ because $F_q$ is exact by \cite[cor.\ 3.9]{HJ-JLMS}.  Since $F_q\mu$ is a monomorphism, this shows $F_q\Ker \mu = 0$ whence $\Ker \mu \cong C_qF_q \Ker \mu \cong 0$ by \cite[lem.\ 7.28(a)]{HJ-JLMS}, where $C_q$ is the functor of \cite[prop.\ 7.15]{HJ-JLMS}.  So $\mu$ is a monomorphism.
\end{proof}

\begin{Lemma}
\label{lem:20_21}
Let $q$ in $Q_0$ and $X$ in ${}_{ Q,A }\!\Mod$ be given.  
\begin{enumerate}
\setlength\itemsep{4pt}

  \item  The following set of $A$-left submodules of $E_qX$ is non-empty and has a maximal element with respect to inclusion.
\[
  \cM = 
  \Bigg\{ M \subseteq E_qX
  \:\Bigg|\! 
  \begin{array}{l}
    \mbox{the inclusion morphism $M \xrightarrow{} E_qX$ has an adjoint} \\[1mm]
    \mbox{morphism $F_qM \xrightarrow{} X$ which is a monomorphism}
  \end{array}
  \Bigg\}
\]

  \item  Suppose that $E_qX$ is in $\Inj( A )$.  Then so is each maximal element of $\cM$.

\end{enumerate}
\end{Lemma}

\begin{proof}
(i)  The set $\cM$ is non-empty because it contains $M = 0$.  We will use Zorn's Lemma to prove that $\cM$ has a maximal element, so suppose that a totally ordered subset $\cI$ of $\cM$ is given; we must prove that $\cI$ has an upper bound in $\cM$.

There is a small filtered category $I$ whose objects are the modules in $\cI$ and whose morphisms are the inclusions between modules in $\cI$.  There is a functor $I \xrightarrow{ M } \Mod( A )$ acting as the identity on objects and morphisms, and the colimit of $M$ is
\[
  C = \bigcup_{ i \in \cI } M( i ).
\]
We will prove that $C$ is in $\cM$ whence it is clearly an upper bound for $\cI$ in $\cM$.  That is, we will prove that the inclusion morphism $C \xrightarrow{ \gamma } E_qX$ has an adjoint morphism $F_qC \xrightarrow{ \varphi } X$ which is a monomorphism.

For each morphism $i \xrightarrow{ \alpha } j$ in $I$ there is a commutative diagram
\[
\vcenter{
  \xymatrix @+0.5pc {
    & & C \ar^{ \gamma }[dd] \\
    M( i ) \ar^{ M( \alpha ) }[r] \ar@/^1.3pc/^{ \iota_i }[urr] \ar@/^-1.3pc/_{ \mu_i }[drr] & M( j ) \ar@/^0.4pc/^<<<<<{ \iota_j }[ur] \ar@/_0.4pc/_-{ \mu_j }[dr] \\
    & & E_qX
                    }
        }
\]
where all arrows are inclusions.  The universal cone to $C$ is $\{ M( i ) \xrightarrow{ \iota_i } C \}_{ i \in \cI }$, and the cone $\{ M( i ) \xrightarrow{ \mu_i } E_qX \}_{ i \in \cI }$ induces the inclusion morphism $C \xrightarrow{ \gamma } E_qX$.  Remark \ref{rmk:adjunction}(ii) gives an induced commutative diagram
\[
\vcenter{
  \xymatrix @+1.5pc {
    & & F_qC \ar^{ \varphi }[dd] \\
    F_q \big( M( i ) \big) \ar^{ F_q ( M( \alpha ) ) }[r] \ar@/^1.5pc/^{ F_q( \iota_i ) }[urr] \ar@/^-1.5pc/_{ \varphi_i }[drr] & F_q \big( M( j ) \big) \ar@/^0.5pc/^{ F_q( \iota_j ) }[ur] \ar@/_0.5pc/_{ \varphi_j }[dr] \\
    & & X
                    }
        }
\]
where $\varphi_i$, $\varphi_j$, $\varphi$ are the adjoint morphisms of $\mu_i$, $\mu_j$, $\gamma$.  The functor $F_q$ is a left adjoint hence preserves colimits, so $\{ F_q \big( M( i ) \big) \xrightarrow{ F_q( \iota_i ) } F_qC \}_{ i \in \cI }$ is the universal cone to the colimit of $F_q \circ M$.  The last diagram shows that $\{ F_q \big( M( i ) \big) \xrightarrow{ \varphi_i } X \}_{ i \in \cI }$ is a cone inducing the adjoint morphism $F_qC \xrightarrow{ \varphi } X$.  Since the $M( i )$ are in $\cM$, the $\varphi_i$ are monomorphisms.  Since ${}_{ Q,A }\!\Mod$ is a Grothendieck abelian category, filtered colimits preserve monomorphisms, so $\varphi$ is a monomorphism as desired.

(ii)  Let $M \subseteq E_qX$ be a maximal element of $\cM$.  Since $E_qX$ is in $\Inj( A )$, to prove that $M$ is in $\Inj( A )$ we will assume $M \subseteq N \subseteq E_qX$ with $M$ essential in $N$ and prove $M = N$; this is sufficient by \cite[lem.\ V.2.2 and prop.\ V.2.4]{Stenstroem}.

Let $M \xrightarrow{ m } N$ be the inclusion and consider Remark \ref{rmk:adjunction}(ii).  There is a commutative diagram \eqref{equ:lem:pre20_21:a:a} where $\mu$ and $\nu$ are the inclusions into $E_qX$, and the remark gives the commutative diagram \eqref{equ:lem:pre20_21:a:b} where $\varphi$ and $\psi$ are the adjoint morphisms of $\mu$ and $\nu$.  Assume $M \subsetneq N$.  Since $M$ is maximal in $\cM$, the morphism $\varphi$ is a monomorphism but the morphism $\psi$ is not.  But then Diagram \eqref{equ:lem:pre20_21:a:b} contradicts that $F_qM \xrightarrow{ F_qm } F_qN$ is an essential extension by Lemma \ref{lem:19}.
\end{proof}

\section{Main theorems}
\label{sec:main}

This section provides different characterisations of minimal semiinjective objects in Theorem \ref{thm:15_23_25} and uses them to prove Theorems \ref{thm:22}, \ref{thm:27}, \ref{thm:36}, which were stated in the introduction.  Not all parts of Theorem \ref{thm:15_23_25} are required for the subsequent proofs, but we consider them worthwhile in their own right.  The adjoint functors $E_q$ and $F_q$ in parts (v)--(vii) were defined in Remark \ref{rmk:EFG}.

\begin{Theorem}
\label{thm:15_23_25}
Let $I$ be a semiinjective object in ${}_{ Q,A }\!\Mod$.  The following conditions are equivalent.
\begin{enumerate}
\setlength\itemsep{4pt}

  \item  $I$ is minimal in the sense of Definition \ref{def:A}(ii), that is, if $J \subseteq I$ with $J$ in ${}_{ Q,A }\!\Inj$, then $J = 0$.
  
  \item  If $E \subseteq I$ with $E$ in $\cE$, then $E = 0$.

  \item  Each weak equivalence $I \xrightarrow{} X$ in ${}_{ Q,A }\!\Mod$ is a split monomorphism.

  \item  If an endomorphism $I \xrightarrow{ f } I$ in ${}_{ Q,A }\!\Mod$ induces an automorphism in $\cD_Q( A )$, then $f$ is already an automorphism.

  \item  For $q$ in $Q_0$ and $D$ in $\Inj( A )$, if a monomorphism $D \xrightarrow{} E_qI$ satisfies that the adjoint morphism $F_qD \xrightarrow{} I$ is a monomorphism, then $D = 0$.
  
  \item  For $q$ in $Q_0$ and $M$ in $\Mod( A )$, if a monomorphism $M \xrightarrow{} E_qI$ satisfies that the adjoint morphism $F_qM \xrightarrow{} I$ is a monomorphism, then $M = 0$.
  
  \item  For $q$ in $Q_0$ and $M$ in $\Mod( A )$, if a monomorphism $M \xrightarrow{ \mu } E_qI$ satisfies 
\[
  \Image( F_qM \xrightarrow{ F_q\mu } F_qE_q I ) \cap Z_qI = 0,
\]
  then $M = 0$.  Here we write
\[
  Z_qX = \Ker( F_qE_qX \xrightarrow{ \varepsilon_X } X )
\]
for $X$ in ${}_{ Q,A }\!\Mod$, where $\varepsilon$ is the counit of the adjoint pair $( F_q,E_q )$.

\end{enumerate}
\end{Theorem}

\begin{proof}
Before starting the proof proper, we recall from \cite[thm.\ 6.5]{HJ-JLMS} that $\cE^{ \perp }$ is a Frobenius category with projective-injective objects ${}_{ Q,A }\!\Inj$, and that there is an equivalence
\[
  \cD_Q( A ) \cong \frac{ \cE^{ \perp } }{ {}_{ Q,A }\!\Inj }.
\]  
The right hand side is the naive quotient category, which has the same objects as $\cE^{ \perp }$ and $\Hom$ spaces obtained by dividing by the subspaces of morphisms factoring through an object of ${}_{ Q,A }\!\Inj$.  Equivalently, the $\Hom$ spaces are obtained by dividing by the equivalence relation ``$\sim$''.  Hence condition (iv) can be replaced by
\begin{itemize}
\setlength\itemsep{4pt}

  \item[(iv')]  If an endomorphism $I \xrightarrow{ f } I$ in ${}_{ Q,A }\!\Mod$ induces an automorphism in $\frac{ \cE^{ \perp } }{ {}_{ Q,A }\!\Inj }$, then $f$ is already an automorphism.

\end{itemize}
See also Proposition \ref{pro:weq_property}(i).

(i) $\Rightarrow$ (ii):  Let $E \subseteq I$ with $E$ in $\cE$ be given.  Since ${}_{ Q,A }\!\Mod$ is a Grothendieck abelian category, there is a ${}_{ Q,A }\!\Inj$-envelope $E \xrightarrow{} J$, see \cite[prop.\ V.2.5 and cor.\ X.4.3]{Stenstroem}.  It is an $\cE^{ \perp }$-envelope by Lemma \ref{lem:9}(ii), so we can factorise as follows where the vertical arrow is the inclusion.
\[
\vcenter{
  \xymatrix @+0.5pc {
    E \ar[d] \ar[r] & J \ar^{ j }[dl] \\
    I \\
                    }
        }
\]
Since $E \xrightarrow{} J$ is an essential extension, $j$ is a monomorphism.  Identifying $J$ with its image under $j$, we have $E \subseteq J \subseteq I$.  But then $J = 0$ by (i) and $E = 0$ follows.

(ii) $\Rightarrow$ (vi):  Let $M \xrightarrow{} E_qI$ be a monomorphism whose adjoint morphism $F_qM \xrightarrow{} I$ is a monomorphism.  Since $F_qM$ is in $\cE$ by \cite[lem.\ 7.14 and thm.\ 7.1]{HJ-JLMS} we have $F_qM = 0$ by (ii).  But then $M \cong C_qF_qM = 0$ by \cite[lem.\ 7.28(a)]{HJ-JLMS}, where $C_q$ is the functor of \cite[prop.\ 7.15]{HJ-JLMS}.  

(vi) $\Rightarrow$ (v) is clear.

(v) $\Rightarrow$ (i):  Let $J \subseteq I$ with $J$ in ${}_{ Q,A }\!\Inj$ be given.  Combining \cite[proof of lem.\ 7.29]{HJ-JLMS} and \cite[lem.\ 3.4 and prop.\ 3.7]{HJ-TAMS} we can write $J$ up to isomorphism as $\coprod_{ p \in Q_0 } F_pD_p$ where each $D_p$ is in $\Inj( A )$.  There is a commutative diagram for each $q$ in $Q_0$,
\[
\vcenter{
  \xymatrix @+0.5pc {
    F_qD_q \ar^-{ \iota_q }[r] \ar_{ \varphi_q }[d] & \coprod_{ p \in Q_0 } F_pD_p  \lefteqn{,} \ar^{ j }[dl] \\
    I \\
                    }
        }
\]
where $\iota_q$ denotes the coproduct inclusion, $j$ the inclusion of $J$ into $I$.  Since $\iota_q$ and $j$ are monomorphisms, so is $F_qD_q \xrightarrow{ \varphi_q } I$.  It is the adjoint morphism of a morphism $D_q \xrightarrow{} E_qI$ which is a monomorphism by Lemma \ref{lem:14}.  But then $D_q = 0$ by (v).  This holds for each $q$ in $Q_0$, so $J = 0$.

(i) $\Rightarrow$ (iv'): This part of the proof is divided into three steps.

Step 1: Assume that an endomorphism $I \xrightarrow{ f } I$ in ${}_{ Q,A }\!\Mod$ induces the identity morphism in $\frac{ \cE^{ \perp } }{ {}_{ Q,A }\!\Inj }$.  We will prove that $f$ is a monomorphism.

The assumption means that there are morphisms $I \xrightarrow{ a } J \xrightarrow{ b } I$ with $J$ in ${}_{ Q,A }\!\Inj$ such that $\id_I - f = ba$.  Composing with the inclusion $\Ker f \xrightarrow{ k } I$ gives $( \id_I - f )k = bak$, that is, $k = bak$.  Since $k$ is a monomorphism, so is $bak$, and hence so is $ak$.  By \cite[proof of prop.\ 2.5]{Stenstroem} there is a commutative diagram
\[
\vcenter{
  \xymatrix @+0.5pc {
    \Ker f \ar^-{ \kappa }[r] \ar_{ ak }[d] & J' \ar^{ j' }[dl] \\
    J \\
                    }
        }
\]
where $\kappa$ is a ${}_{ Q,A }\!\Inj$-envelope, $j'$ the inclusion of a subobject, and this gives $bj'\kappa = bak = k$.  Since $k$ is a monomorphism, so is $bj'\kappa$, and hence so is $bj'$ since $\kappa$ is an essential extension.  So $bj'$ lets us view $J'$ as a subobject of $I$ whence $J' = 0$ by (i).  Hence $\Ker f = 0$ and $f$ is a monomorphism as claimed.

Step 2: Assume that an endomorphism $I \xrightarrow{ f } I$ in ${}_{ Q,A }\!\Mod$ induces an automorphism in $\frac{ \cE^{ \perp } }{ {}_{ Q,A }\!\Inj }$.  We will prove that $f$ is a monomorphism.

Pick $I \xrightarrow{ g } I$ such that $g$ induces an inverse of $f$ in $\frac{ \cE^{ \perp } }{ {}_{ Q,A }\!\Inj }$.  Then $gf$ induces the identity morphism in $\frac{ \cE^{ \perp } }{ {}_{ Q,A }\!\Inj }$ whence $gf$ is a monomorphism by Step 1.  Hence $f$ is a monomorphism.

Step 3: Assume that an endomorphism $I \xrightarrow{ f } I$ in ${}_{ Q,A }\!\Mod$ induces an automorphism in $\frac{ \cE^{ \perp } }{ {}_{ Q,A }\!\Inj }$.  We will prove that $f$ is an automorphism. 

By Step 2 we know that $f$ is a monomorphism, so there is a short exact sequence
\begin{equation}
\label{equ:thm:15_23_25:a}
  0 \xrightarrow{} I \xrightarrow{ f } I \xrightarrow{} J \xrightarrow{} 0,
\end{equation}  
which induces an exact sequence
\[
  \Ext_{ Q,A }^1( E,I )
  \xrightarrow{}
  \Ext_{ Q,A }^1( E,J )
  \xrightarrow{}
  \Ext_{ Q,A }^2( E,I )
\]
for each $E$.  If $E$ is in $\cE$, then the outer terms are zero.  This is true for the first term because $I$ is in $\cE^{ \perp }$.  For the third term, it is true because $I$ is in $\cE^{ \perp }$ while $( \cE,\cE^{ \perp } )$ is a hereditary cotorsion pair by \cite[thm.\ 4.4(b)]{HJ-JLMS}.  Hence the middle term is zero, so $J$ is in $\cE^{ \perp }$.  Thus, \eqref{equ:thm:15_23_25:a} is a short exact sequence with terms in $\cE^{ \perp }$, so induces a triangle in the triangulated category $\frac{ \cE^{ \perp } }{ {}_{ Q,A }\!\Inj }$ by \cite[lem.\ I.2.7]{Happel-book}.  Since $f$ induces an automorphism, $J$ must induce the zero object whence $J$ is in ${}_{ Q,A }\!\Inj$.  But then $J$ is projective-injective in the Frobenius category $\cE^{ \perp }$ so \eqref{equ:thm:15_23_25:a} is split exact.  Up to isomorphism, $J$ is hence a subobject of $I$ so $J$ is zero by (i).  So \eqref{equ:thm:15_23_25:a} proves that $f$ is an automorphism.

(iv') $\Rightarrow$ (iii): Let $I \xrightarrow{ i } X$ be a weak equivalence.  It follows from Proposition \ref{pro:weq_property}(iii) that there exists a morphism $X \xrightarrow{ \xi } I$ such that $\xi i$ induces the identity morphism in $\frac{ \cE^{ \perp } }{ {}_{ Q,A }\!\Inj }$.  Hence $\xi i$ is an automorphism by (iv').  If $\theta$ is the inverse then $\theta\xi i = \id_I$, which shows that $i$ is a split monomorphism.

(iii) $\Rightarrow$ (i): Let $J \subseteq I$ with $J$ in ${}_{ Q,A }\!\Inj$ be given.  There is an induced short exact sequence $0 \xrightarrow{} J \xrightarrow{} I \xrightarrow{ f } X \xrightarrow{} 0$, and $J$ is in $\cE$ by \cite[thm.\ 4.4(b)]{HJ-JLMS} so $f$ is a weak equivalence by Proposition \ref{pro:10}(ii).  But then $f$ is a split monomorphism by (iii) whence $J = 0$.  

(vi) $\Leftrightarrow$ (vii): By Remark \ref{rmk:adjunction}(i), a morphism $M \xrightarrow{ \mu } E_qI$ has the adjoint morphism $F_qM \xrightarrow{} I$ defined as the composition of the morphisms $F_qM \xrightarrow{ F_q\mu } F_qE_qI \xrightarrow{ \varepsilon_I } I$, where $\varepsilon$ is the counit of the adjoint pair $( F_q,E_q )$.  When $\mu$ is a monomorphism, so is $F_q\mu$ by \cite[cor.\ 3.9]{HJ-JLMS}, so the condition
\[
  \Image( F_qM \xrightarrow{ F_q\mu } F_qE_q I ) \cap Z_qI = 0
\]
is equivalent to the condition that the adjoint morphism is a monomorphism.  Hence (vi) and (vii) express the same condition.
\end{proof}

\begin{proof}
[Proof of Theorem \ref{thm:22}]
(ii) For each $q$ in $Q_0$, Lemma \ref{lem:20_21}(i) says there is a submodule $D_q \subseteq E_qI$ maximal with respect to the property that the inclusion $D_q \xrightarrow{ \delta_q } E_qI$ has an adjoint morphism $F_qD_q \xrightarrow{ \varphi_q } I$ which is a monomorphism.  The module $D_q$ is in $\Inj( A )$ by Lemma \ref{lem:20_21}(ii) because $E_qI$ is in $\Inj( A )$ by \cite[thm.\ E]{HJ-TAMS}.

We claim that the object
\[
  J' = \coprod_{ q \in Q_0 } F_qD_q
\]
is in ${}_{ Q,A }\!\Inj$.  To see so, observe that $J'$ can be written $\prod_{ q \in Q_0 } G_{ Sq }D_q$ by \cite[lem.\ 3.4 and prop.\ 3.7]{HJ-TAMS}, and that $G_{ Sq }D_q$ is in ${}_{ Q,A }\!\Inj$ by \cite[lem.\ 3.11]{HJ-JLMS} since $D_q$ is in $\Inj( A )$.  There is a unique morphism $\varphi'$ such that the following diagram is commutative for each $q$ in $Q_0$,
\[
\vcenter{
  \xymatrix @+0.5pc {
    F_qD_q \ar_{ \iota_q }[d] \ar^{ \varphi_q }[dr] \\
    J' \ar_-{ \varphi' }[r] & I \lefteqn{,} \\
                    }
        }
\]
where $\iota_q$ denotes the coproduct inclusion.  Combining with Remark \ref{rmk:adjunction}(i) provides the following commutative diagram.
\[
\vcenter{
  \xymatrix @+0.5pc {
    F_qD_q \ar@/^2.5pc/[rr]^{ \varphi_q } \ar_{ \iota_q }[d] \ar^{ F_q\delta_q }[r] & F_qE_qI \ar^-{ \varepsilon_I }[r] & I \ar@{=}[d] \\
    J' \ar_-{ \varphi' }[rr] && I \\
                    }
        }
\]
The morphism $\varphi'$ is a monomorphism by Lemma \ref{lem:18}.  Since $J'$ is in ${}_{ Q,A }\!\Inj$, the morphism $\varphi'$ is a split monomorphism which can be viewed as the inclusion of a direct summand.  The diagram shows that
\begin{equation}
\label{equ:thm:22:a}
  \begin{array}{l}
    \mbox{ the image of the monomorphism $\varphi_q$ is contained in } \\
    \mbox{ the direct summand $J'$ for each $q$ in $Q_0$. }
  \end{array}
\end{equation}

Consider the complement $I'$ of $J'$ in $I$.  Then $I \cong I' \oplus J'$, so it is clear that $I'$ is semiinjective.  To complete the proof, we will prove that $I'$ is minimal semiinjective by proving that it satisfies the condition in Theorem \ref{thm:15_23_25}(v).  So let $q$ in $Q_0$ and $D$ in $\Inj( A )$ be given and assume that a monomorphism $D \xrightarrow{ \delta' } E_qI'$ has an adjoint morphism $F_qD \xrightarrow{ \psi' } I'$ which is a monomorphism.  The inclusion $I' \xrightarrow{ i' } I$ is a split monomorphism, hence so is $E_qI' \xrightarrow{ E_q i' } E_qI$.  The composition $\delta$ of the morphisms $D \xrightarrow{ \delta' } E_qI' \xrightarrow{ E_q i' } E_qI$ is a monomorphism, and there is a commutative diagram
\[
\vcenter{
  \xymatrix @+0.5pc {
    F_qD \ar@/^2.5pc/[rr]^{ \psi' } \ar^-{ F_q\delta' }[r] \ar@{=}[d] & F_qE_qI' \ar^-{ \varepsilon_{ I' } }[r] \ar^{ F_qE_q i' }[d] & I' \ar^{ i' }[d] \\
    F_qD \ar@/^-2.5pc/[rr]_{ \psi } \ar_-{ F_q\delta }[r] & F_qE_qI \ar_-{ \varepsilon_I }[r] & I \\
                    }
        }
\]
where $\psi'$ and $\psi$ are the adjoint morphisms of $\delta'$ and $\delta$, see Remark \ref{rmk:adjunction}(i).  Since $\psi'$ and $i'$ are monomorphisms, so is $\psi = i'\psi'$.  The diagram shows that 
\begin{equation}
\label{equ:thm:22:b}
  \begin{array}{l}
    \mbox{ the image of the monomorphism $\psi$ is contained in } \\
    \mbox{ the direct summand $I'$, which is the complement of $J'$ in $I$. }
  \end{array}
\end{equation}

Now consider the morphism $D_q \oplus D \xrightarrow{ ( \delta_q,\delta ) } E_qI$.  Its adjoint morphism is the composition of the morphisms
\[
\vcenter{
  \xymatrix @+0.5pc {
    F_qD_q \oplus F_qD \ar^-{ ( F_q\delta_q,F_q\delta ) }[rr] && F_qE_qI \ar^-{ \varepsilon_I }[rr] && I, \\
                    }
        }
\]
so its adjoint morphism is $( \varphi_q,\psi )$ which is a monomorphism by Equations \eqref{equ:thm:22:a} and \eqref{equ:thm:22:b}.  Hence $( \delta_q,\delta )$ is a monomorphism by Lemma \ref{lem:14}.  However, by the maximality of $D_q$, this implies $\delta = 0$, and since $\delta$ is a monomorphism, it follows that $D = 0$.  Hence we have proved that $I'$ satisfies the condition in Theorem \ref{thm:15_23_25}(v).

(i) By \cite[thm.\ 5.9]{HJ-JLMS} there is a complete cotorsion pair $( \cE,\cE^{ \perp } )$ in the sense of \cite[def.\ 2.2.1 and lem.\ 2.2.6]{Goebel-Trlifaj-book}.  Hence there is a short exact sequence $0 \xrightarrow{} X \xrightarrow{ x } I \xrightarrow{} E \xrightarrow{} 0$ with $I$ in $\cE^{ \perp }$ and $E$ in $\cE$.  By part (ii) of the theorem we have $I = I' \oplus J'$ with $I'$ a minimal semiinjective object and $J'$ in ${}_{ Q,A }\!\Inj$.  Hence there is a short exact sequence $0 \xrightarrow{} J' \xrightarrow{} I \xrightarrow{ i } I' \xrightarrow{} 0$.  Note that $J'$ is in $\cE$ by \cite[thm.\ 4.4(b)]{HJ-JLMS}.  The morphisms $X \xrightarrow{ x } I$ and $I \xrightarrow{ i } I'$ are weak equivalences by Proposition \ref{pro:10}, so the composition $X \xrightarrow{ ix } I'$ is a weak equivalence by \cite[prop.\ 5.12]{Hovey_MathZ}, hence a minimal semiinjective resolution.
\end{proof}

\begin{proof}
[Proof of Theorem \ref{thm:27}]
Let $I \xrightarrow{ i } I'$ be a weak equivalence between minimal semiinjective objects.  Theorem \ref{thm:15_23_25}(iii) says that $i$ is a split monomorphism, so there exists a split epimorphism $I' \xrightarrow{ i' } I$ such that $i'i = \id_I$.  But then \cite[prop.\ 5.12]{Hovey_MathZ} implies that $i'$ is a weak equivalence, so Theorem \ref{thm:15_23_25}(iii) implies that $i'$ is a split monomorphism.  In particular, $i'$ is an epimorphism and a monomorphism, hence an isomorphism. 
\end{proof}

\begin{proof}
[Proof of Theorem \ref{thm:36}]
By Proposition \ref{pro:weq_property}(iii) the morphism $X \xrightarrow{ x } I$ induces a bijection
\[
\vcenter{
  \xymatrix @+0.5pc {
    \Hom_{ Q,A }( I,I' )/\sim \ar[r] & \Hom_{ Q,A }( X,I' )/\sim \lefteqn{.} \\
                    }
        }
\]
This implies the two first bullet points of part (ii).  To prove the third bullet point, observe that the relation $ix \sim x'$ in ${}_{ Q,A }\!\Mod$ induces an equality in $\cD_Q( A )$.  This implies that $i$ induces an isomorphism in $\cD_Q( A )$ because the weak equivalences $x$ and $x'$ induce isomorphisms in $\cD_Q( A )$.  But then $i$ is a weak equivalence in ${}_{ Q,A }\!\Mod$ by \cite[thm.\ 1.2.10(iv)]{Hovey_Book}, and then $i$ is an isomorphism in ${}_{ Q,A }\!\Mod$ by Theorem \ref{thm:27}.
\end{proof}

\section{Differential modules}
\label{sec:differential}

This section proves Theorem \ref{thm:RZ}, which was stated in the introduction.  Our theory can be specialised to the theory of differential modules by setting
\[
  \Bk = \BZ
\]
and setting $Q$ equal to the $\Bk$-preadditive category given by
\[
  \xymatrix{
    q \ar@(ul,ur)^{ \partial }
           }
\]
with $\partial^2 = 0$, and we will do so in this section.  Then a $Q$-shaped diagram is a differential module as introduced on page \pageref{page:Diff}, so ${}_{ Q,A }\!\Mod$ is equal to $\Diff( A )$, the category of differential modules over $A$.  Note that this $Q$ does satisfy the assumptions of the introduction with pseudoradical given by $\fr_q = \Bk \cdot \partial$.  We now explain how some concepts from the theory of ${}_{ Q,A }\!\Mod$ specialise to $\Diff( A )$; see also \cite[A.2]{HJ-TAMS}.

\begin{altbfhpg}
[\bf Cohomology functors]
\label{altbfhpg:Homology_functors}
specialise as
\[
  \BH_{ [q] }^i \rightsquigarrow \H
\]
for each $i \geqslant 1$, where $\H$ is the homology functor on differential modules introduced on page \pageref{page:Diff}.  This can be proved by computing $\BH_{ [q] }^i( - )$, which by \cite[def.\ 7.11]{HJ-JLMS} is $\Ext_Q^i( S\langle q \rangle,- )$, using the projective resolution $\cdots \xrightarrow{} Q( q,- ) \xrightarrow{} Q( q,- )$ of $S\langle q \rangle$.
\end{altbfhpg}

\begin{altbfhpg}
[\bf Weak equivalences]
\label{altbfhpg:weq}
specialise as
\[
  \weq \rightsquigarrow \{\, \mu \,|\, \mbox{$\mu$ is a quasi-isomorphism} \,\},
\]
where a morphism $\mu$ of differential modules is a {\em quasi-isomorphism} if $\H( \mu )$ is an isomorphism.  This follows from \ref{altbfhpg:Homology_functors} and \cite[thm.\ 7.2]{HJ-JLMS}.
\end{altbfhpg}

\begin{altbfhpg}
[\bf The class $\cE$ of exact objects]
\label{altbfhpg:The_class_E}
specialises as
\[
  \cE \rightsquigarrow \{\, ( M,\partial ) \in \Diff( A ) \,|\,  \mbox{$( M,\partial )$ is exact} \,\},
\]
where a differential module $( M,\partial )$ is {\em exact} if $\H( M,\partial ) = 0$.  This follows from \ref{altbfhpg:Homology_functors} and \cite[thm.\ 7.1]{HJ-JLMS}.
\end{altbfhpg}

\begin{altbfhpg}
[\bf The class $\cE^{ \perp }$ of semiinjective objects]
\label{altbfhpg:The_class_Eperp}
will only be specialised when the left global dimension of $A$ is finite.  Then
\[
  \cE^{ \perp } \rightsquigarrow \{\, ( J,\partial ) \in \Diff( A ) \,|\, \mbox{$J$ is in $\Inj( A )$} \,\}
\]
by \cite[thm.\ E]{HJ-TAMS}, and the right hand class can be written as
\[
  \{\, ( J,\partial ) \in \Diff( A ) \,|\, J \mbox{ is Gorenstein injective in } \Mod( A ) \,\}
\]
by the dual of \cite[prop.\ 2.27]{Holm-GHD}.  By \cite[thm.\ 1.1]{Wei} the last class can be written as
\[
  \{\, ( J,\partial ) \in \Diff( A ) \,|\, \mbox{$( J,\partial )$ is Gorenstein injective in $\Diff( A )$} \,\}.
\]
\end{altbfhpg}

\begin{altbfhpg}
[\bf The class of minimal semiinjective objects]
\label{altbfhpg:Minimal_semiinjective}
will only be specialised when the left global dimension of $A$ is finite.  
Then
\[
  \Bigg\{ I
  \:\Bigg|\! 
  \begin{array}{l}
    \mbox{$I$ is minimal} \\[1mm]
    \mbox{semiinjective}
  \end{array}
  \Bigg\}
  \rightsquigarrow
  \Bigg\{ ( J,\partial ) \in \Diff( A )
  \:\Bigg|\! 
  \begin{array}{l}
    \mbox{$( J,\partial )$ is Gorenstein injective without} \\[1mm]
    \mbox{non-zero injective summands in $\Diff( A )$}
  \end{array}
  \Bigg\}
\]
by \ref{altbfhpg:The_class_Eperp} and Definition \ref{def:A}(ii).  Here, the right hand class modulo isomorphism is the left hand class in Theorem \ref{thm:RZ}.
\end{altbfhpg}

\begin{altbfhpg}
[\bf Minimal semiinjective resolutions]
\label{altbfhpg:Minimal_semiinjective_resolution}
will only be specialised when the left global dimension of $A$ is finite.  Then they become quasi-isomorphisms $( M,\partial_M ) \xrightarrow{} ( J,\partial_J )$ where $( J,\partial_J )$ is Gorenstein injective without non-zero injective summands in $\Diff( A )$.  This follows from \ref{altbfhpg:weq}, \ref{altbfhpg:Minimal_semiinjective}, and Definition \ref{def:A}(iv).
\end{altbfhpg}

Recall from page \pageref{page:Diff} the functors $\B$, $\Z$, $\H$ from $\Diff( A )$ to $\Mod( A )$ which send a differential module to its boundaries, cycles, and homology.  There is a short exact sequence in $\Mod( A )$,
\begin{equation}
\label{equ:BZH}
  0 
  \xrightarrow{} \B( M,\partial )
  \xrightarrow{} \Z( M,\partial )
  \xrightarrow{ \zeta } \H( M,\partial )
  \xrightarrow{} 0,
\end{equation}
natural with respect to $( M,\partial )$ in $\Diff( A )$.

The modules $\B( M,\partial )$, $\Z( M,\partial )$, $\H( M,\partial )$ can be viewed as differential modules $\big( \B( M,\partial ),0 \big)$, $\big( \Z( M,\partial ),0 \big)$, $\big( \H( M,\partial ),0 \big)$ with zero differential.  There is a canonical short exact sequence in $\Diff( A )$,
\begin{equation}
\label{equ:j}
  0
  \xrightarrow{} \big( \Z( M,\partial ),0 \big)
  \xrightarrow{ j } ( M,\partial )
  \xrightarrow{} \big( \B( M,\partial ),0 \big)
  \xrightarrow{} 0,
\end{equation}
natural with respect to $( M,\partial )$ in $\Diff( A )$, and $\H( j )$ can be identified with the morphism $\zeta$ in the sequence \eqref{equ:BZH}.

\begin{Lemma}
\label{lem:MA9}
For $( M,\partial_M )$ in $\Diff( A )$, consider the differential module $\big( \H( M,\partial_M ),0 \big)$ with zero differential.

Assume that the sequence \eqref{equ:BZH} in $\Mod( A )$ is split exact.  Then there exists a monic quasi-isomorphism $\big( \H( M,\partial_M ),0 \big) \xrightarrow{ \eta } ( M,\partial_M )$ in $\Diff( A )$.
\end{Lemma}

\begin{proof}
Since the sequence \eqref{equ:BZH} is split exact, there is a splitting morphism $p$ giving the following diagram.
\begin{equation}
\label{equ:MA9:a}
\vcenter{
  \xymatrix @+0.5pc {
    0 \ar[r] & \B( M,\partial_M ) \ar[r] & \Z( M,\partial_M ) \ar[r]^{ \zeta } \ar@/^1.35pc/[l]^{ p } & \H( M,\partial_M ) \ar[r] & 0. \\
                    }
        }
\end{equation}
We can also view $p$ as a morphism $\big( \Z( M,\partial_M ),0 \big) \xrightarrow{ p } \big( \B( M,\partial_M ),0 \big)$ and use it to construct the following diagram in $\Diff( A )$,
\begin{equation}
\label{equ:MA9:d}
\vcenter{
  \xymatrix @+0.5pc {
    0 \ar[r] & \big( \Z( M,\partial_M ),0 \big) \ar[r]^-{ j } \ar[d]_{ p } & ( M,\partial_M ) \ar[r] \ar[d]^{ m } & \big( \B( M,\partial_M ),0 \big) \ar[r] \ar@{=}[d] & 0 \\
    0 \ar[r] & \big( \B( M,\partial_M ),0 \big) \ar[r] & ( V,\partial_V ) \ar[r] & \big( \B( M,\partial_M ),0 \big) \ar[r] & 0 \lefteqn{,} & {}\save[]-<9.15cm,-0.925cm>*\txt<8pc>{pushout} \restore \\
                    }
        }
\end{equation}
where the first row is the short exact sequence \eqref{equ:j}, the first square is a pushout square, and the second row is short exact; see \cite[prop.\ VIII.4.2]{MacLane-book}.  The Snake Lemma implies $\Ker p \cong \Ker m$ and $\Coker p \cong \Coker m$.  Since $p$ is the splitting morphism from diagram \eqref{equ:MA9:a}, we have $\Ker p \cong \big( \H( M,\partial_M ),0 \big)$ and $\Coker p = 0$.  Combining this information provides a short exact sequence in $\Diff( A )$,
\[
  0
  \xrightarrow{} \big( \H( M,\partial_M ),0 \big)
  \xrightarrow{ \eta } ( M,\partial_M )
  \xrightarrow{ m } ( V,\partial_V )
  \xrightarrow{} 0.
\]
To prove that $\eta$ is a quasi-isomorphism, it is enough to prove $\H( V,\partial_V ) = 0$ by Proposition \ref{pro:10}(i), \ref{altbfhpg:weq}, and \ref{altbfhpg:The_class_E}.

The pushout square in \eqref{equ:MA9:d} induces a short exact sequence
\begin{equation}
\label{equ:MA9:b}
  0
  \xrightarrow{} \big( \Z( M,\partial_M ),0 \big)
  \xrightarrow{ \tiny \begin{pmatrix} j \\ -p \end{pmatrix} } ( M,\partial_M ) \oplus \big( \B( M,\partial_M ),0 \big)
  \xrightarrow{} ( V,\partial_V )
  \xrightarrow{} 0
\end{equation}
in $\Diff( A )$ by \cite[prop.\ 2.53]{Freyd-book}.  Note that $\begin{pmatrix} j \\ -p \end{pmatrix}$ is indeed a monomorphism since $j$ is a monomorphism.  As remarked before the proposition, $\H( j )$ can be identified with $\zeta$ from diagrams \eqref{equ:BZH} and \eqref{equ:MA9:a}, and $\H( -p )$ can clearly be identified with $\Z( M,\partial_M ) \xrightarrow{ -p } \B( M,\partial_M )$, so $\H \begin{pmatrix} j \\ -p \end{pmatrix} = \begin{pmatrix} \H( j ) \\ \H( -p ) \end{pmatrix}$ can be identified with $\begin{pmatrix} \zeta \\ -p \end{pmatrix}$.  This is an isomorphism since $p$ is a splitting morphism, see diagram \eqref{equ:MA9:a}.  Hence the long exact homology sequence induced by \eqref{equ:MA9:b} implies $\H( V,\partial_V ) = 0$ as desired.
\end{proof}

\begin{proof}
[Proof of Theorem \ref{thm:RZ}]
Since $A$ is left hereditary, it has finite left global dimension, so paragraphs \ref{altbfhpg:Homology_functors} through \ref{altbfhpg:Minimal_semiinjective_resolution} apply.

It is clear that the homology functor $\Diff( A ) \xrightarrow{ \H } \Mod( A )$ induces a map $\H$ as shown in Equation \eqref{equ:RZ:a}.  We will prove that an inverse map, $\operatorname{K}$, is given by mapping the isomorphism class of $M$ in $\Mod( A )$ to the isomorphism class of $( J,\partial_J )$ in $\Diff( A )$ where $( M,0 ) \xrightarrow{ \mu } ( J,\partial_J )$ is a quasi-isomorphism and $( J,\partial_J )$ is Gorenstein injective without non-zero injective summands in $\Diff( A )$.  

Such a $\mu$ is a minimal semiinjective resolution by \ref{altbfhpg:Minimal_semiinjective_resolution}, so it exists by Theorem \ref{thm:22}(i) and $( J,\partial_J )$ is determined up to isomorphism by Theorem \ref{thm:36}.

The map $\operatorname{K}$ takes values in the left hand set of Equation \eqref{equ:RZ:a} by construction.

$\H\!\operatorname{K} = \id$: The quasi-isomorphism $( M,0 ) \xrightarrow{ \mu } ( J,\partial_J )$ provides the second equality in the following computation up to isomorphism.
\[
  \H\!\operatorname{K}( M ) = \H( J,\partial_J ) = \H( M,0 ) = M
\]

$\operatorname{K}\!\H = \id$:  Let $( J,\partial_J )$ be Gorenstein injective without non-zero injective summands in $\Diff( A )$.  By \ref{altbfhpg:The_class_Eperp} we have $J$ in $\Inj( A )$ whence the quotient $\B( J,\partial_J )$ of $J$ is also in $\Inj( A )$ because $A$ is left hereditary.  This implies that the sequence \eqref{equ:BZH} is split exact, so Lemma \ref{lem:MA9} gives a quasi-isomorphism $\big( \H( J,\partial_J ),0 \big) \xrightarrow{ \eta } ( J,\partial_J )$, and by the definition of $\operatorname{K}$ this shows
\[
  \operatorname{K}\!\H( J,\partial_J ) = ( J,\partial_J ).
\qedhere
\]
\end{proof}

\appendix

%
\numberwithin{equation}{section}
\renewcommand{\theequation}{\Alph{section}.\arabic{equation}}

\renewcommand{\thesubsection}{\Alph{section}.\Roman{subsection}}

\renewcommand{\thesection}{\greek{section}}
\section{Minimal semiinjective resolutions in the classic derived category}
\label{app:complexes}

Our theory can be specialised to the theory of minimal semiinjective resolutions in $\cD( A )$ by setting $Q$ equal to the $\Bk$-preadditive category given by Figure \ref{fig:linear_quiver} modulo the relations that any two arrows compose to zero, and we will do so in this appendix.  Then a $Q$-shaped diagram is a chain complex, so ${}_{ Q,A }\!\Mod$ is equal to $\Ch( A )$, the category of chain complexes and chain maps over $A$, and $\cD_Q( A )$ is equal to $\cD( A )$.  Theorems \ref{thm:22}, \ref{thm:27}, \ref{thm:36}, and \ref{thm:15_23_25} specialise to the following results due to \cite{AFH}, \cite[app.\ B]{CFH-book}, \cite[sec.\ 10]{HBF-preprints-19}, \cite[sec.\ 2.3 and 2.4]{Garcia-Rozas-book}, \cite[app.\ B]{Krause}.

{\bf Theorem \ref{thm:22} for complexes.}
{\em 
\begin{enumerate}
\setlength\itemsep{4pt}

  \item  Each $X$ in $\Ch( A )$ has a minimal semiinjective resolution.

  \item  Each semiinjective complex $I$ in $\Ch( A )$ has the form $I = I' \oplus J'$ in $\Ch( A )$ with $I'$ a minimal semiinjective complex and $J'$ a null homotopic complex of injective modules.  

\end{enumerate}
}

{\bf Theorem \ref{thm:27} for complexes.}
{\em 
If $I \xrightarrow{ i } I'$ in $\Ch( A )$ is a quasi-isomorphism between minimal semiinjective complexes, then $i$ is an isomorphism in $\Ch( A )$.
}

{\bf Theorem \ref{thm:36} for complexes.}
{\em 
If $X \xrightarrow{ x } I$ and $X \xrightarrow{ x' } I'$ are minimal semiinjective resolutions in $\Ch( A )$, then:
\begin{itemize}
\setlength\itemsep{4pt}

  \item  The diagram
\[
\vcenter{
  \xymatrix @+0.5pc {
    X \ar^{ x }[r] \ar_{ x' }[d] & I \ar@{.>}^{ i }[dl] \\
    I' \\
                    }
        }
\]
can be completed with a chain map $i$ such that $ix$ is chain homotopic to $x'$ in $\Ch( A )$.

  \item  The chain map $i$ is unique up to chain homotopy.

  \item  Each completing chain map $i$ is an isomorphism in $\Ch( A )$.

\end{itemize}
}

{\bf Theorem \ref{thm:15_23_25} for complexes.}
{\em 
Let
\[
  I = 
  \cdots \xrightarrow{}
  I_2 \xrightarrow{ \partial_2 }
  I_1 \xrightarrow{ \partial_1 }
  I_0 \xrightarrow{ \partial_0 }
  I_{-1} \xrightarrow{ \partial_{-1} }
  I_{-2} \xrightarrow{}
  \cdots
\]
be a semiinjective complex in $\Ch( A )$.  The following conditions are equivalent.
\begin{enumerate}
\setlength\itemsep{4pt}

  \item  $I$ is minimal in the sense that if $J \subseteq I$ with $J$ a null homotopic complex of injective modules, then $J = 0$.
  
  \item  If $E \subseteq I$ with $E$ an exact complex, then $E = 0$.

  \item  Each quasi-isomorphism $I \xrightarrow{} X$ in $\Ch( A )$ is a split monomorphism.

  \item  If an endomorphism $I \xrightarrow{ f } I$ in $\Ch( A )$ induces an automorphism in $\cD( A )$, then $f$ is already an automorphism.

  \item  $\Ker \partial_q$ is an essential submodule of $I_q$ for each $q$.

\end{enumerate}
}

The specialisations are obtained by applying Figure \ref{fig:specialisation}, which explains how some concepts from ${}_{ Q,A }\!\Mod$ specialise to $\Ch( A )$.  Note that items (i)-(iv) in Theorem \ref{thm:15_23_25} specialise to items (i)-(iv) in Theorem \ref{thm:15_23_25} for complexes, while items (v)-(vii) in Theorem \ref{thm:15_23_25}  all specialise to item (v) in Theorem \ref{thm:15_23_25} for complexes.  
\begin{figure}
\begin{tabular}{c|c|c}
      & ${}_{ Q,A }\!\Mod$
      & $\Ch( A )$ \\[2mm] \cline{1-3}
  (a) & $\cD_Q( A )$
      & $\cD( A )$ \vphantom{$x^{x^{x^{x^{x^x}}}}$} \\[1.5mm]
  (b) & $\cE$
      & exact complexes \\[1.5mm]
  (c) & $\weq$
      & quasi-isomorphisms \\[1.5mm]
  (d) & semiinjective object
      & semiinjective complex \\[1.5mm]
  (e) & semiinjective resolution
      & semiinjective resolution \\[1.5mm]
  (f) & minimal semiinjective resolution
      & minimal semiinjective resolution \\[1.5mm]
  (g) & ${}_{ Q,A }\!\Inj$
      & null homotopic complexes of injective modules \\[1.5mm]
  (h) & $\sim$ in the category $\cE^{ \perp }$
      & chain homotopy of chain maps \\[1.5mm]
\end{tabular}
\caption{Set $Q$ equal to the $\Bk$-preadditive category given by Figure \ref{fig:linear_quiver} modulo the relations that any two arrows compose to zero.  Then ${}_{ Q,A }\!\Mod$ is equal to $\Ch( A )$.  This table explains how some concepts specialise.  Items (a)-(e) are given by \cite[2.4, 3.2, 4.2]{HJ-Abel}, item (g) is \cite[exer.\ 14.8]{Kashiwara-Schapira}, and item (h) follows from item (g).  Item (f) holds because Theorem \ref{thm:15_23_25}(i), characterising minimal semiinjective objects in ${}_{ Q,A }\!\Mod$, specialises to \cite[prop.\ 2.3.14(b)]{Garcia-Rozas-book}, characterising minimal semiinjective complexes.}
\label{fig:specialisation}
\end{figure}
For instance, consider Theorem \ref{thm:15_23_25}(vi).  The functors $F_q$ and $E_q$ from Remark \ref{rmk:EFG} specialise to
\[
\xymatrix
{
  \Ch( A )
    \ar[rr]_{ E_q } &&
  \Mod( A ),
    \ar@/_1.5pc/[ll]_{ F_q }
}
\]
given on objects by
\[
  F_qM = \cdots \xrightarrow{} 0 \xrightarrow{} M \xrightarrow{ \id } M \xrightarrow{} 0 \xrightarrow{} \cdots
  \;\;\;,\;\;\;
  E_qI = I_q.	
\]
In $F_qM$, the module $M$ is placed in homological degrees $q$ and $q-1$.  A morphism $M \xrightarrow{ \mu } E_qI$ has the adjoint morphism $F_qM \xrightarrow{ \varphi } I$ given by the following chain map.
\[
\vcenter{
  \xymatrix @+0.5pc {
    F_qM \ar@<3ex>_{ \varphi }[d] \!\!\!\!\!\!\!\!\!\!\!\!\!\!\!\! & = \; \cdots \ar[r] & 0 \ar[r] \ar[d] & M \ar^{ \id }[r] \ar_{ \mu }[d] & M \ar[r] \ar^{ \partial_q\mu }[d] & 0 \ar[r] \ar[d] & \cdots \\
    I \!\!\!\!\!\!\!\!\!\!\!\!\!\!\!\! & = \; \cdots \ar[r] & I_{ q+1 } \ar[r] & I_q \ar_{ \partial_q }[r] & I_{ q-1 } \ar[r] & I_{ q-2 } \ar[r] & \cdots \\
                    }
        }
\]
Hence Theorem \ref{thm:15_23_25}(vi) specialises to the statement that given a monomorphism $M \xrightarrow{ \mu } I_q$ for which $M \xrightarrow{ \partial_q\mu } I_{ q-1 }$ is also a monomorphism, we must have $\mu = 0$.  That is, given a monomorphism $M \xrightarrow{ \mu } I_q$ for which $\Image \mu \cap \Ker \partial_q = 0$, we must have $\mu = 0$.  This is equivalent to item (v) in Theorem \ref{thm:15_23_25} for complexes.

\medskip
\noindent
{\bf Acknowledgement.}
This work was supported by a DNRF Chair from the Danish National Research Foundation (grant DNRF156), by a Research Project 2 from the Independent Research Fund Denmark (grant 1026-00050B), and by Aarhus University Research Foundation (grant AUFF-F-2020-7-16).

\end{document}